\documentclass [twoside,reqno, 12pt] {amsart}

\usepackage{tikz}

\usepackage{amsfonts}
\usepackage{amssymb}
\usepackage{a4}
\usepackage{color}

\newtheorem{thm}{Theorem}[section]

\newtheorem{lem}[thm]{Lemma}
\newtheorem{prop}[thm]{Proposition}

\newtheorem{rem}[thm]{Remark}

\theoremstyle{definition}

\numberwithin{equation}{section}

\renewcommand{\Re}{\hbox{Re}\,}

\newcommand{\C}{\mathbb{C}}

\newcommand{\N}{\mathbb{N}}

\newcommand{\R}{\mathbb{R}}

\newcommand{\Z}{\mathbb{Z}}

\def\tilde{\widetilde}
\def \bfo {\begin {eqnarray*} }
\def \efo {\end {eqnarray*} }
\def \ba {\begin {eqnarray*} }
\def \ea {\end {eqnarray*} }
\def \beq {\begin {eqnarray}}
\def \eeq {\end {eqnarray}}

\def \det {\hbox{det}}

\def \p {\partial}

\def\tilde{\widetilde}
\def \bfo {\begin {eqnarray*} }
\def \efo {\end {eqnarray*} }
\def \ba {\begin {eqnarray*} }
\def \ea {\end {eqnarray*} }
\def \beq {\begin {eqnarray}}
\def \eeq {\end {eqnarray}}

\def \det {\hbox{det}}

\def \p {\partial}

\begin{document}

 \title[Absolute continuity of the periodic Schr\"odinger operator]{Absolute continuity of the periodic Schr\"odinger operator in transversal geometry}

\author[Krupchyk]{Katsiaryna Krupchyk}

\address
         {K. Krupchyk, Department of Mathematics\\
University of California, Irvine\\ 
CA 92697-3875, USA }

\email{katya.krupchyk@uci.edu}

\author[Uhlmann]{Gunther Uhlmann}

\address
       {G. Uhlmann, Department of Mathematics\\
       University of Washington\\
       Seattle, WA  98195-4350\\
       USA\\
       Department of Mathematics and Statistics \\
       University of Helsinki\\
         P.O. Box 68 \\
         FI-00014   Helsinki\\
         Finland}
\email{gunther@math.washington.edu}

\maketitle

\begin{abstract}
We show that the spectrum of a Schr\"odinger operator on $\R^n$, $n\ge 3$, with a periodic smooth Riemannian metric, whose  conformal multiple has a product structure with one Euclidean direction, and with a periodic electric potential in $L^{n/2}_{\text{loc}}(\R^n)$, is purely absolutely continuous. Previously known results in the case of a general metric are obtained in  \cite{Friedlander_2002}, see also  \cite{Filonov_ Tikhomirov},  under the assumption that the metric, as well as the potential, are reflection symmetric.

\end{abstract}

\section{Introduction}
Consider the Schr\"odinger operator,
\[
H=-\Delta_g+q\quad \text{on}\quad \R^n, \quad n\ge 3.
\]
Here $-\Delta_g$ is the Laplace--Beltrami operator, associated to a $C^\infty$ smooth  Riemannian metric $g$, given by 
\[
-\Delta_g=|g|^{-1/2}D_{x_j}(|g|^{1/2} g^{jk}D_{x_k}),
\]
where $D_{x_j}=i^{-1}\p_{x_j}$,  $(g^{jk})$ is the matrix inverse of $(g_{jk})$, and $|g|=\det(g_{jk})$. Throughout the paper we  shall assume that the metric $g$ and the electric potential $q$  are $2\pi$--periodic in all variables.

Let $q\in L^{n/2}_{\text{loc}}(\R^n)$ be real--valued.  Then  the operator $H$ is the self-adjoint operator on $L^2(\R^n; |g|^{1/2}dx)$, given via the closed semibounded from below sesquilinear form,
\begin{equation}
\label{eq_int_form}
h[u,v]=\int_{\R^n} g^{jk}D_{x_k} u \overline{D_{x_j} v}|g|^{1/2} dx+\int_{\R^n} qu \overline{v}|g|^{1/2}dx,
\end{equation}
with the domain $\mathcal{D}(h)=H^1(\R^n)$,  the standard $L^2$ based Sobolev space, see Appendix \ref{sec_1}.

Starting with the pioneering work \cite{Thomas_1973}, the structure of the spectrum of the periodic Schr\"odinger operator $H$ in $\R^n$ has been intensively studied.  We refer to  \cite{Reed_Simon_book_4}, \cite{Birman_Suslina_1998}, \cite{Birman_Suslina_1999},  \cite{Shen_2001}, and \cite{Shen_Zhao_2008} for  some of the works in this direction.  In particular, in the case of the Euclidean metric, i.e. $g=I$, we know that the spectrum of $H$ is purely absolutely continuous for a potential $q\in L^{n/2}_{\text{loc}}(\R^n)$ when $n\ge 3$,  thanks to the works \cite{Birman_Suslina_1999} and \cite{Shen_2001}, and for $q\in L^{1+\varepsilon}_{\text{loc}}(\R^2)$, $\varepsilon>0$,  thanks to the work \cite{Birman_Suslina_1998}.  These results can be extended to the case of a metric $g$ conformal to the Euclidean, i.e. $g=cI$,  where $c>0$ is a smooth periodic function.

The absolute continuity of the spectrum for the magnetic Schr\"odinger operator with periodic electric and magnetic potentials, in the case of the Euclidean metric, was established in \cite{Birman_Suslina_1997} in two dimensions and in \cite{Sobolev_1999} in the higher dimensions. See also \cite{Hempel_Herbst_1995}.

In two dimensions, the case of a general $C^\infty$--smooth metric $g$ was investigated completely in \cite{Morame_1998} and the absolute continuity of the spectrum of a periodic magnetic Schr\"odinger operator was established in \cite{Morame_1998}.  

In higher dimensions, the case of  a general  metric is wide open and the only result concerning the absolute continuity of the spectrum of a periodic magnetic Schr\"odinger operator that we are aware of is due to \cite{Friedlander_2002}, under the assumption that the operator is invariant under reflection $x_1\mapsto -x_1$, in the case of smooth coefficients.  The smoothness assumptions in the result of \cite{Friedlander_2002} were relaxed in \cite{Filonov_ Tikhomirov}, and the absolute continuity of the spectrum was obtained for a Lipschitz continuous metric $g$, the magnetic potential $A\in L^{n+\varepsilon}_{\text{loc}}(\R^n)$, $\varepsilon>0$, and the electric potential $q\in L^{n/2}_{\text{loc}}(\R^n)$, 
see also \cite{Filonov_Sobolev_2013}.

The purpose of this paper is to consider the case $n\ge 3$ and to show the absolute continuity of the spectrum of the periodic Schr\"odinger operator $H$ with a Riemannian metric $g$, whose conformal multiple has a product structure with one Euclidean direction,
and $q\in L^{n/2}_{\text{loc}}(\R^n)$. To be precise, we assume that  
\begin{equation}
\label{eq_metric_transv}
g(x_1,x')=c(x)\begin{pmatrix}
1 & 0\\
0& g_0(x')
\end{pmatrix},
\end{equation}
where $c>0$ is a positive smooth function, $x=(x_1,x')\in \R^n$, and $g_0$ is a Riemannian metric on $\R^{n-1}$.  

Our main result is as follows. 
\begin{thm}
\label{thm_main}
Let  $g$ be a  $C^\infty$--smooth Riemannian metric on $\R^n$, $n\ge 3$,  of the form \eqref{eq_metric_transv}, and let  $q\in L^{n/2}_{\emph{\text{loc}}}(\R^n; \C)$.  Assume that $g$ and $q$ are periodic with respect to the lattice $2\pi\Z^n$.  Then  the Schr\"odinger operator $H=-\Delta_g+q$ in $L^2(\R^n; |g|^{1/2}dx)$, defined by the sesquilinear form \eqref{eq_int_form}, has no eigenvalues. In the case when $q$ is  real--valued, the spectrum of $H$ is purely absolutely continuous. 
\end{thm}

\begin{rem}
The metric $c^{-1} g$ is independent of $x_1$ and therefore satisfies the assumptions of \cite{Friedlander_2002}. On the other hand,  no symmetry condition is imposed on the conformal factor $c$  and the potential $q$.  
\end{rem}

Our inspiration for considering metrics of the form \eqref{eq_metric_transv} came from the recent works on inverse boundary value problems for Schr\"odinger operators on compact Riemannian manifolds with boundary, equipped with metrics of this form, see \cite{Dos_Santos_F_Kenig_Salo_Uhlmann}.

We would like to mention that the problem of absolute continuity of the spectrum of the Schr\"odinger operator $H$ on a smooth cylinder $M\times \R^m$ was treated in \cite{Kachkov_Filonov_2010}, see also references given there. Here 
$M$ is a smooth compact Riemannian manifold, and the metric $g$ on $M\times \R^m$ is  a product of a Riemannian metric on $M$ and the Euclidean metric on $\R^m$. Furthermore, the potential $q$ is assumed to be periodic with respect to the Euclidean variables. In the case when the dimension of the cylinder is $\ge 3$, the absolute continuity is established in \cite{Kachkov_Filonov_2010} when $q\in L^{n/2+\varepsilon}_{\text{loc}}(M\times \R^m)$, $\varepsilon >0$.

Let us finish the introduction by making some indications concerning the main steps in the proof of Theorem \ref{thm_main}. 
First, as a consequence of general spectral theory, it will be seen that it is sufficient to treat the case when the conformal factor $c$ in \eqref{eq_metric_transv} satisfies $c=1$, and therefore to work with the operator $D_{x_1}^2-\Delta_{g_0(x')}+q$. We would like to show the absence of eigenvalues for this operator, and replacing $q$ by $q-\lambda$ we reduce the problem to establishing that zero is not an eigenvalue. An application of the Floquet theory combined with the Thomas approach, see Proposition \ref{prop_Floquet_Thomas} below, allows us next to conclude that it suffices to find  $\theta\in \C^n$ such that 
\[
\text{Ker}\, (H(\theta))=\{0\}.
\]
Here the operator $H(\theta)$, acting on the torus $\mathbb{T}^n=\R^n/2\pi\Z^n$, is given by
\begin{align*}
H(\theta)=|g|^{-1/2}(D_{x_j}+\theta_j)(|g|^{1/2}  g^{jk}(D_{x_k}+\theta_k))+q, \quad 1\le j,k\le n.
\end{align*}
See subsection \ref{subsec_A2} for the definition of this operator using the method of quadratic forms.

We shall make the following choice of the complex quasimomentum $\theta$, for $\tau\in \R$,
\[
\theta_\tau=\bigg(\frac{1}{2}+i\tau,0,\dots, 0 \bigg)\in \C^n,
\]
with the corresponding family of operators given by 
\[
H(\theta_\tau)=\bigg(D_{x_1}+\frac{1}{2}\bigg)^2+2i\tau \bigg(D_{x_1}+\frac{1}{2}\bigg)-\tau^2 -\Delta_{g_0(x')}+q.
\]

In the case when $q\in L^\infty(\R^n)$, the fact that we know explicitly the eigenvalues of the one--dimensional normal operator $\big(D_{x_1}+\frac{1}{2}\big)^2+2i\tau \big(D_{x_1}+\frac{1}{2}\big)-\tau^2$ on $\mathbb{T}^1$ implies that for $\tau\in \R$ with $|\tau|$ sufficiently large, we have  
\[
\frac{|\tau|}{2} \|u\|_{L^2(\mathbb{T}^n)}\le  \| H(\theta_\tau)u\|_{L^2(\mathbb{T}^n)},
\]
for $u\in \mathcal{D}(H (\theta_\tau) )$, and therefore, $\text{Ker}\,(H(\theta_\tau))=\{0\}$. 

In the case when $q\in L^{n/2}_{\text{loc}}(\R^n)$, we shall show that there exists a constant $C>0$ such that for $\tau\in \R$ with $|\tau|$ sufficiently large, the following estimate holds, 
\begin{equation}
\label{eq_int_L_p_est}
 \|u\|_{L^{\frac{2n}{n-2}}(\mathbb{T}^n)}\le  C \|H(\theta_\tau) u  \|_{L^{\frac{2n}{n+2}}(\mathbb{T}^n)},
\end{equation}
when $u\in \mathcal{D}(H(\theta_\tau))$, and thus, $\text{Ker}\,(H(\theta_\tau))=\{0\}$.

When establishing \eqref{eq_int_L_p_est}, the crucial ingredients are spectral cluster estimates for the  non-negative elliptic self-adjoint operator  
\[
\bigg(D_{x_1}+\frac{1}{2}\bigg)^2-\Delta_{g_0(x')},
\] 
acting on  $L^2(\mathbb{T}^n)$,   see \cite{Sogge_1988},  \cite{Seeger_Sogge_1989},   \cite{Sogge_book}, and  uniform resolvent estimates for it, obtained recently in   \cite{Dos_Santos_F_Kenig_Salo_resolvent},  \cite{Bourgain_Shao_Sogge_Yao}, \cite{Krup_Uhlmann_resolvent}.  

Let us point out that the idea of using the spectral cluster estimates of \cite{Sogge_1988}, and uniform $L^p$ resolvent estimates for constant coefficient elliptic operators on the torus, in the study of absolute continuity of the spectrum of the Schr\"odinger operator goes back to the work \cite{Shen_2001}, where the absolute continuity of the spectrum of a 
Schr\"odinger operator with a Euclidean metric and a potential in $L^{n/2}_{\text{loc}}$ was established.  The spectral cluster estimates of  \cite{Sogge_1988} were also used in \cite{Kachkov_Filonov_2010}. In this paper we use the recently established uniform $L^p$ resolvent estimates for elliptic self-adjoint operators with variable coefficients. 

The paper is organized as follows.  After explaining how to get rid of the conformal factor in the metric in Section  \ref{sec_conf},   Section \ref{sec_3} is devoted to the proof of Theorem \ref{thm_main} in the special case of a bounded potential.  The proof in this case is quite straightforward and is presented here as a warm-up, before handling the general case in Section \ref{sec_4}. Appendix  \ref{sec_1} contains some standard material pertaining to the definition of our operators, review of the Floquet theory, and a description of the Thomas approach to the absolute continuity problem. It is presented merely for the convenience of the reader.

\section{Removing the conformal factor}

\label{sec_conf}
In the case when $q\in L^{\frac{n}{2}}_{\text{loc}}(\R^n)$ is real--valued,  it follows from Remark \ref{rem_app_selfadj} in Appendix \ref{sec_1} that the singular continuous component of the spectrum of the Schr\"odinger operator $H$ is empty, and the pure point spectrum is at most discrete, consisting only of isolated points without finite accumulation points. To establish the absolute continuity of $H$ it suffices therefore to show the absence of eigenvalues. Hence,  in what follows we shall concentrate on proving the absence of eigenvalues in the general non-self-adjoint case, i.e. when $q$ is complex--valued.

We have the following conformal relation, see \cite{Dos_Santos_F_Kenig_Salo_Uhlmann},
\[
c^{\frac{n+2}{4}} (-\Delta_g +q)(c^{\frac{-(n-2)}{4}} u)= (-\Delta_{c^{-1}g}  +q_c) u,
\]
for $u\in \mathcal{D}$. Here 
\[
q_c=cq+ c^{\frac{n+2}{4}}(-\Delta_g)(c^{-\frac{(n-2)}{4}})\in L^{\frac{n}{2}}_{\text{loc}}(\R^n)
\]
is $2\pi\Z^n$--periodic, and
\begin{align*}
\mathcal{D}&:=\mathcal{D}(-\Delta_{c^{-1}g}  +q_c)=\{u\in H^1(\R^n): c^{\frac{n+2}{4}} (-\Delta_g +q)c^{\frac{-(n-2)}{4}}u\in L^2(\R^n)\}\\
&=
\{u\in H^1(\R^n): (g^{jk} D_{x_j} D_{x_k}+q)u\in L^2(\R^n)\}.
\end{align*}
The last equality follows since $c\in C^\infty(\R^n)$ is a strictly positive periodic function.

Assume that $\lambda\in \C$ is an eigenvalue of $-\Delta_g+q$. Then we have  for some $u\in \mathcal{D}$, not vanishing identically, 
\begin{equation}
\label{eq_conf_transf}
0 = c^{\frac{(n+2)}{4}} (-\Delta_g+q-\lambda)u=(-\Delta_{c^{-1}g}+q_{c} -c \lambda)(c^{\frac{n-2}{4}} u),
\end{equation}
and therefore $0$ is an eigenvalue of the operator $-\Delta_{c^{-1}g}+q_{c} -c \lambda$.  

To establish the absence of eigenvalues of the operator $H$ it is thus sufficient to prove that zero is not an eigenvalue of the operator $-\Delta_{c^{-1}g}+q$ with an arbitrary periodic $q\in  L^{\frac{n}{2}}_{\text{loc}}(\R^n)$. In what follows we shall therefore assume that the metric $g$ is of the form \eqref{eq_metric_transv}, where the conformal  factor $c=1$. 

It follows from Proposition \ref{prop_Floquet_Thomas} in Appendix \ref{sec_1} that to prove that  zero is not an eigenvalue of the operator $-\Delta_{g}+q$ it suffices to show that  there exists $\theta\in \C^n$ such that the operator $H(\theta)$, defined in \eqref{eq_2_2_0}, is injective.

Let $\tau\in \R$, and let us set
\[
\theta_\tau=\bigg(\frac{1}{2}+i\tau,0,\dots, 0 \bigg)\in \C^n,
\]
and 
\[
H_0(\theta_\tau)=\bigg(D_{x_1}+\frac{1}{2}\bigg)^2+2i\tau \bigg(D_{x_1}+\frac{1}{2}\bigg)-\tau^2 -\Delta_{g_0(x')},
\]
so that
\[
H(\theta_\tau)=H_0(\theta_\tau)+q.
\]
Theorem \ref{thm_main} will follow once we prove that the operator $H(\theta_\tau)$ is injective for $\tau\in \R$ with $|\tau|$ sufficiently large.

\section{Proof of Theorem \ref{thm_main} in the case of a bounded potential}

\label{sec_3}

\begin{prop} 
\label{prop_carleman_L_infty}
For all $\tau\in \R$ with $|\tau|\ge 1$, the following estimate holds,
\begin{equation}
\label{eq_Carleman_1}
|\tau| \|u\|_{L^2(\mathbb{T}^n)}\le \| H_0 (\theta_\tau) u  \|_{L^2(\mathbb{T}^n)},
\end{equation}
for $u\in H^{2}(\mathbb{T}^n)$. 
\end{prop}

\begin{proof}
By a density argument it suffices to prove the estimate \eqref{eq_Carleman_1} for  $u\in C^\infty(\mathbb{T}^n)$. Expanding $u$ in the Fourier series with respect to $x_1$, we have
\[
u(x_1,x')=\sum_{j\in \Z} e^{ijx_1} u_j(x'),
\]
where 
\[
u_j(x')=\frac{1}{2\pi} \int_0^{2\pi} u(y_1,x')e^{-ijy_1}dy_1,
\]
and therefore we get
\begin{align*}
H_0(\theta_\tau)u&=\bigg( \bigg(D_{x_1}+\frac{1}{2}\bigg)^2+2i\tau \bigg(D_{x_1}+\frac{1}{2}\bigg)-\tau^2 - \Delta_{g_0(x')}\bigg) u\\
&=
\sum_{j\in \Z}  \bigg( \bigg(j+\frac{1}{2}\bigg)^2+2i\tau \bigg(j+\frac{1}{2}\bigg)-\tau^2 -\Delta_{g_0(x')} \bigg) e^{ijx_1} u_j(x').
\end{align*}

Since the operator $-\Delta_{g_0(x')}$, acting on $L^2(\mathbb{T}^{n-1})$, is self-adjoint, we have 
\[
\|(-\Delta_{g_0(x')}-z)^{-1}\|_{L^2(\mathbb{T}^{n-1})\to L^2(\mathbb{T}^{n-1})}\le \frac{1}{|\text{Im}\, z|}, \quad \text{Im}\, z\ne 0,
\]
and hence, as  $|j+1/2|\ge 1/2$, $j\in \Z$, for $|\tau|\ge 1$, we get
\[
\bigg\|\bigg( \bigg(j+\frac{1}{2}\bigg)^2+2i\tau \bigg(j+\frac{1}{2}\bigg)-\tau^2 -\Delta_{g_0(x')} \bigg)^{-1}\bigg\|_{L^2(\mathbb{T}^{n-1})\to L^2(\mathbb{T}^{n-1})}\le  \frac{1}{|\tau|}.
\]

By Parseval's identity, we obtain that 
\begin{align*}
\frac{1}{2\pi }\| H_0(\theta_\tau)u &\|^2_{L^2(\mathbb{T}^n)}\\
&=\sum_{j\in \Z} \bigg\| \bigg( \bigg(j+\frac{1}{2}\bigg)^2+2i\tau \bigg(j+\frac{1}{2}\bigg)-\tau^2 -\Delta_{g_0(x')} \bigg) u_j(x')\bigg\|^2_{L^2(\mathbb{T}^{n-1})}\\
&\ge \frac{1}{2\pi} |\tau|^2\| u\|^2_{L^2(\mathbb{T}^n)}.
\end{align*}
 The proof of Proposition \ref{prop_carleman_L_infty} is complete. 
\end{proof}

Let $q\in L^\infty(\mathbb{T}^n)$, so that $\mathcal{D}(H (\theta_\tau) )=H^2(\mathbb{T}^n)$.  Thus, by Proposition \ref{prop_carleman_L_infty}  we conclude that for $|\tau|\ge 1$ sufficiently large, 
\[
\frac{|\tau|}{2} \|u\|_{L^2(\mathbb{T}^n)}\le  \| (H_0 (\theta_\tau)+q) u\|_{L^2(\mathbb{T}^n)},
\]
for $u\in \mathcal{D}(H (\theta_\tau) )$.   The proof of  Theorem \ref{thm_main} in the case $q\in L^\infty$ is therefore complete.

\section{Proof of Theorem \ref{thm_main} in the case of a potential $q\in L^{\frac{n}{2}}(\mathbb{T}^n)$ }
\label{sec_4}

Let us start by recalling the following chain of continuous inclusions, where the first and the last
ones follow from the Sobolev embedding theorem, 
\[
H^1(\mathbb{T}^n)\hookrightarrow L^{\frac{2n}{n-2}}(\mathbb{T}^n) \hookrightarrow L^{2}(\mathbb{T}^n) \hookrightarrow L^{\frac{2n}{n+2}}(\mathbb{T}^n) \hookrightarrow H^{-1}(\mathbb{T}^n).
\]

We shall need the following result. 
\begin{lem}
\label{lem_domain}
\label{lem_domain_emb_L_p}
 Let $q\in L^{\frac{n}{2}}(\mathbb{T}^n)$.  Then  
 \[
 \mathcal{D}(H(\theta_\tau))\subset  W^{2,\frac{2n}{n+2}}(\mathbb{T}^n).
 \] 
\end{lem}

\begin{proof}

Let $u\in \mathcal{D}(H(\theta_\tau))$. Then $f:=H(\theta_\tau)u\in L^2(\mathbb{T}^n)$.  Using the fact that $\mathcal{D}(H(\theta_\tau))\subset H^1(\mathbb{T}^n)$, Sobolev's embedding $H^1(\mathbb{T}^n)\hookrightarrow L^{\frac{2n}{n-2}}(\mathbb{T}^n)$, and H\"older's inequality, 
\[
\|qu\|_{L^{\frac{2n}{n+2}}(\mathbb{T}^n)}\le \|q\|_{L^{\frac{n}{2}}(\mathbb{T}^n)} \|u\|_{L^{\frac{2n}{n-2}}(\mathbb{T}^n)},
\] 
we get $qu\in L^{\frac{2n}{n+2}}(\mathbb{T}^n)$. Hence, we have
\[
H_0(\theta_\tau)u=f-qu\in L^{\frac{2n}{n+2}}(\mathbb{T}^n).
\]
As $u\in L^{\frac{2n}{n+2}}(\mathbb{T}^n)$ and the operator $H_0(\theta_\tau)$ is elliptic with smooth coefficients, by elliptic regularity, we conclude that $u\in W^{2,\frac{2n}{n+2}}(\mathbb{T}^n)$.  The proof  is complete. 
\end{proof}

\begin{prop}
\label{prop_carleman_L_n_2}
There  exists a constant $C>0$ such that for all $\tau\in \R$ with $|\tau|$ sufficiently large, the following estimates hold,
\begin{equation}
\label{eq_L_p_L_2}
\|u\|_{L^2(\mathbb{T}^n)}\le \frac{C}{|\tau|^{1/2} } \| H_0(\theta_\tau)  u \|_{L^{\frac{2n}{n+2}}(\mathbb{T}^n)},
\end{equation}
and
\begin{equation}
\label{eq_L_p_L_p}
\|u\|_{L^{\frac{2n}{n-2}}(\mathbb{T}^n)}\le C \| H_0 (\theta_\tau) u  \|_{L^{\frac{2n}{n+2}}(\mathbb{T}^n)},
\end{equation}
 when $u\in W^{2, \frac{2n}{n+2}}(\mathbb{T}^n)$.
\end{prop}

\begin{proof}
 Here we use the notation and some ideas of \cite{Dos_Santos_F_Kenig_Salo_resolvent}. We denote by  $0=\lambda_0< \lambda_1\le \lambda_2\le \dots$ the sequence of eigenvalues of $-\Delta_{g_0(x')}$ on $\mathbb{T}^{n-1}$, counted with their multiplicities, and by $(\psi_k)_{k\ge 0}$  the corresponding sequence of eigenfunctions forming an orthonormal basis of $L^2(\mathbb{T}^{n-1})$, 
\[
-\Delta_{g_0(x')}\psi_k=\lambda_k\psi_k.
\]

The operator 
\[
H_0(\theta_0)=\bigg(D_{x_1}+\frac{1}{2}\bigg)^2-\Delta_{g_0(x')},
\] 
equipped with the domain $H^2(\mathbb{T}^n)$, is non-negative elliptic self-adjoint on the space $L^2(\mathbb{T}^n; |g|^{1/2}dx)$, whose eigenvalues are given by $(j+\frac{1}{2})^2+\lambda_k$, $j\in \mathbb{Z}$, $k\in \N$,
and the corresponding eigenfunctions are $\tilde \psi_{j,k}(x)=e^{ijx_1}\psi_k(x')$,
i.e. 
\[
H_0(\theta_0)\tilde \psi_{j,k}=\bigg(\bigg(j+\frac{1}{2}\bigg)^2+\lambda_k\bigg)\tilde \psi_{j,k}.
\]
We denote by $\pi_{j,k}:L^2(\mathbb{T}^n)\to L^2(\mathbb{T}^n)$ the orthogonal projection on the linear space, spanned by the eigenfunction $\tilde \psi_{j,k}$,
\[
\pi_{j,k} f(x)=\frac{1}{2\pi} \bigg(\int_{\mathbb{T}^n} f(y) e^{-ijy_1}\overline{\psi_k(y')} \sqrt{|g_0|}dy \bigg) e^{ijx_1} \psi_k(x').
\]
We have 
\[
\sum_{j\in \Z,k\in \N} \pi_{j,k}=I.
\]
Let us denote by $\chi_m$ the spectral projection operator on the space, generated by the eigenfunctions, corresponding to the $m$th spectral cluster of the operator $H_0(\theta_0)$,
\[
\chi_m=\sum_{m\le \sqrt{\big(j+\frac{1}{2}\big)^2+\lambda_k}<m+1} \pi_{j,k},\quad m\in\N.
\]
To establish the estimates \eqref{eq_L_p_L_2} and \eqref{eq_L_p_L_p} we shall need the spectral cluster estimates, obtained in \cite{Sogge_1988}, \cite{Seeger_Sogge_1989}, see also \cite{Sogge_book},
\begin{equation}
\label{eq_Sogges_estim_1}
\|\chi_m f \|_{L^2(\mathbb{T}^n)}\le C(1+m)^{1/2} \|f\|_{L^{\frac{2n}{n+2}}(\mathbb{T}^n)},
\end{equation}
and the dual estimates,
\begin{equation}
\label{eq_Sogges_estim_2}
\|\chi_m f \|_{L^{\frac{2n}{n-2}}(\mathbb{T}^n)}\le C(1+m)^{1/2} \|f\|_{L^2(\mathbb{T}^n)},
\end{equation}
when $f\in C^\infty(\mathbb{T}^n)$.

By a density argument it suffices to  establish the estimates \eqref{eq_L_p_L_2} and \eqref{eq_L_p_L_p} for $u\in C^\infty(\mathbb{T}^n)$.  
Consider the equation 
\[
H_0(\theta_\tau) u=f,
\]
with $u,f\in C^\infty(\mathbb{T}^n)$. Writing $u=\sum_{j,k}\pi_{j,k} u$ and $ f=\sum_{j,k}\pi_{j,k} f$, we get  
\[
\bigg(\bigg(j+\frac{1}{2}\bigg)^2+2i\tau \bigg(j+\frac{1}{2}\bigg)-\tau^2+\lambda_k \bigg)\pi_{j,k} u=\pi_{j,k} f,
\]
with $j\in \Z$ and $k\in \N$.  As $\big|j+\frac{1}{2}\big|\ge \frac{1}{2}$ for $j\in \Z$, we have 
\[
\bigg|\bigg(j+\frac{1}{2}\bigg)^2+2i\tau\bigg(j+\frac{1}{2}\bigg)-\tau^2+\lambda_k \bigg|\ge 2|\tau|\bigg|j+\frac{1}{2}\bigg| \ge |\tau|,
\]
and therefore, when $|\tau|\ge 1$, 
\[
u=G_\tau f:=\sum_{j=-\infty}^\infty \sum_{k=0}^\infty \frac{\pi_{j,k}f}{ \big(j+\frac{1}{2}\big)^2+2i\tau\big(j+\frac{1}{2}\big)-\tau^2+\lambda_k}.
\]
We now come to prove that there exists a constant $C>0$ such that for $\tau\in \R$ with $|\tau|\ge 1$, 
\begin{equation}
\label{eq_L_p_L_2_equiv}
\|G_\tau f\|_{L^2(\mathbb{T}^n)}\le \frac{C}{|\tau|^{1/2}} \|f\|_{L^{\frac{2n}{n+2}}(\mathbb{T}^n)}, 
\end{equation}
when  $f\in C^\infty(\mathbb{T}^n)$.  
Using \eqref{eq_Sogges_estim_1}, we get 
\begin{equation}
\label{eq_3_7_0}
\begin{aligned}
&\|G_\tau f\|_{L^2(\mathbb{T}^n)}^2=\sum_{j=-\infty}^\infty \sum_{k=0}^\infty \frac{\|\pi_{j,k}f\|_{L^2(\mathbb{T}^n)}^2}
{\big| \big(j+\frac{1}{2}\big)^2+2i\tau\big(j+\frac{1}{2}\big)-\tau^2+\lambda_k\big|^2}\\
&\le \sum_{m=0}^\infty \sup_{m\le \sqrt{\big(j+\frac{1}{2}\big)^2+\lambda_k}<m+1}\frac{1}{\big| \big(j+\frac{1}{2}\big)^2+2i\tau\big(j+\frac{1}{2}\big)-\tau^2+\lambda_k|^2}
\|\chi_mf \|_{L^2(\mathbb{T}^n)}^2\\
&\le C S \|f\|_{L^{\frac{2n}{n+2}}(\mathbb{T}^n)}^2,
\end{aligned}
\end{equation}
where 
\begin{equation}
\label{eq_def_Sigma}
S:= \sum_{m=0}^\infty (1+m)\sup_{m\le \sqrt{\big(j+\frac{1}{2}\big)^2+\lambda_k}<m+1}\frac{1}{\big| \big(j+\frac{1}{2}\big)^2+2i\tau\big(j+\frac{1}{2}\big)-\tau^2+\lambda_k\big|^2}.
\end{equation}

Let us show that the  series $S$ converges and behaves as $1/|\tau|$ for $|\tau|$ large. To that end we observe that 
\begin{equation}
\label{eq_3_4}
2\bigg| \bigg(j+\frac{1}{2}\bigg)^2+2i\tau\bigg(j+\frac{1}{2}\bigg)-\tau^2+\lambda_k\bigg|\ge  \bigg|\bigg(j+\frac{1}{2}\bigg)^2+\lambda_k -\tau^2\bigg|+|\tau|.
\end{equation}
Assume now that $m\le \sqrt{\big(j+\frac{1}{2}\big)^2+\lambda_k}<m+1$ and let $m\le |\tau|$. Then using the fact that  $|\tau|\ge 1$, we obtain that 
\begin{equation}
\label{eq_3_5_0}
\begin{aligned}
|m&^2-\tau^2|\le \bigg| \bigg(j+\frac{1}{2}\bigg)^2+\lambda_k -\tau^2\bigg|+\bigg|m^2-  \bigg(j+\frac{1}{2}\bigg)^2-\lambda_k\bigg|\\
&\le \bigg| \bigg(j+\frac{1}{2}\bigg)^2+\lambda_k -\tau^2\bigg|+ 2m+1 \le \bigg| \bigg(j+\frac{1}{2}\bigg)^2+\lambda_k -\tau^2\bigg|+ 3|\tau|,
\end{aligned}
\end{equation}
and therefore, 
\begin{equation}
\label{eq_3_5}
4 \bigg(\bigg| \bigg(j+\frac{1}{2}\bigg)^2+\lambda_k -\tau^2\bigg|+|\tau|\bigg )\ge |m^2-\tau^2|+|\tau|.
\end{equation}
When $m> |\tau|$, we have
\begin{equation}
\label{eq_3_6_0}
\bigg| \bigg(j+\frac{1}{2}\bigg)^2+\lambda_k -\tau^2\bigg|  \ge |m^2-\tau^2|.
\end{equation}
Now we are ready to estimate the sum of the series $S$, given by \eqref{eq_def_Sigma}. Using 
 \eqref{eq_3_4}, \eqref{eq_3_5} and \eqref{eq_3_6_0}, we get  that 
\begin{equation}
\label{eq_3_7}
\begin{aligned}
S
&\lesssim  \sum_{m=0}^\infty \frac{1+m}{(m^2-\tau^2)^2 +\tau^2}\lesssim \frac{1}{|\tau|^4}+\sum_{m=1}^\infty \frac{m}{(m^2-\tau^2)^2 +\tau^2}\\
&\lesssim \frac{1}{|\tau|}+\int_0 ^\infty \frac{t dt}{(t^2-\tau^2)^2+\tau^2}\lesssim \frac{1}{|\tau|} \int_{-\infty}^\infty \frac{ds}{s^2+1}\lesssim\frac{1}{|\tau|}.
\end{aligned}
\end{equation}
Here we have used the fact that the function $t\mapsto t/((t^2-\tau^2)^2+\tau^2)$ is increasing when $t\in [0,|\tau|/\sqrt{3})$ and is decreasing when $t\in (|\tau|/\sqrt{3}, \infty)$, and have performed the change of variables $s=t^2/|\tau|-|\tau|$ in the integral.  Combining  \eqref{eq_3_7_0} and \eqref{eq_3_7}, we obtain \eqref{eq_L_p_L_2_equiv} and therefore, \eqref{eq_L_p_L_2}.

Let us now prove the estimate  \eqref{eq_L_p_L_p} for  $u\in C^\infty(\mathbb{T}^n)$, which amounts to obtaining the following uniform estimate, 
\begin{equation}
\label{eq_L_p_L_p_equiv}
\|G_\tau f\|_{L^{\frac{2n}{n-2}}(\mathbb{T}^n)}\le C \| f\|_{L^{\frac{2n}{n+2}}(\mathbb{T}^n)},
\end{equation}
 for $f\in C^\infty(\mathbb{T}^n)$ and $\tau\in \R$ with $|\tau|$ sufficiently large. 

To that end we shall need the following uniform resolvent estimate for the elliptic self-adjoint operator $H_0(\theta_0)$: for each $\delta\in (0,1)$, there exists a constant $C>0$ such that for all $u\in C^\infty(\mathbb{T}^n)$ and all $z\in \C$ with  $\text{Im}\, z\ge \delta$, we have
\begin{equation}
\label{eq_3_8}
\|u\|_{L^{\frac{2n}{n-2}}(\mathbb{T}^n)}\le C \| ( H_0(\theta_0)-z^2 ) u \|_{L^{\frac{2n}{n+2}}(\mathbb{T}^n)},
\end{equation}  
see \cite{Dos_Santos_F_Kenig_Salo_resolvent}, \cite{Bourgain_Shao_Sogge_Yao} and \cite{Krup_Uhlmann_resolvent}.
When establishing the estimate  \eqref{eq_L_p_L_p_equiv}, we shall follow  \cite{Dos_Santos_F_Kenig_Salo_resolvent} closely, and use a localization argument to  deduce this estimate from  \eqref{eq_3_8}, see also \cite{Kenig_Ruiz_Sogge} and \cite{Shen_2001}. 
We have
\[
f(x_1,x')=\sum_{j\in \Z} e^{ijx_1} f_j(x'), \quad f_j(x')=\frac{1}{2\pi} \int_0^{2\pi} f(y_1,x') e^{-ijy_1}dy_1.
\]
Letting $\chi$ be the characteristic function of the interval $[1/2,1)$, and further localizing $f$ in frequency with respect to the variable $x_1$, we introduce
\begin{align*}
\tilde f_\nu(x_1,x')&= \chi\bigg(\frac{|D_{x_1}|}{2^\nu}\bigg)f(x_1,x')=\sum_{2^{\nu-1}\le |j|< 2^\nu} e^{ijx_1} f_j(x'), \quad \nu= 1,2,\dots,\\
\tilde f_0(x_1,x')&=f_0(x'),
\end{align*}
so that 
\[
f=\sum_{\nu=0}^\infty \tilde f_\nu.
\]
Since the operators $\chi \big(\frac{|D_{x_1}|}{2^\nu}\big)$ and $H_0(\theta_\tau)$ commute, we  observe that the localization of $G_\tau f$ in frequency with respect to the variable $x_1$ is given by 
\[
(\widetilde{G_\tau f})_\nu=G_\tau \tilde f_\nu,\quad \nu=0,1,2,\dots.
\]
Standard arguments based on the one--dimensional Littlewood--Paley theory \cite[Theorem 8.4]{Duoandikoetxea_book},   imply that in order to prove \eqref{eq_L_p_L_p_equiv} it suffices to establish the uniform estimates,
\begin{equation}
\label{eq_3_9}
\|G_\tau \tilde f_\nu \|_{L^{\frac{2n}{n-2}}(\mathbb{T}^n)}\le C \| \tilde f_\nu \|_{L^{\frac{2n}{n+2}}(\mathbb{T}^n)}, \quad \nu=0,1,2,\dots,
\end{equation}
see also \cite[Lemma 2.3]{Krup_Uhlmann_resolvent}. 

When proving \eqref{eq_3_9}, we introduce the resolvent of $H_0(\theta_0)$, given by   
\[
R(\zeta):=( H_0(\theta_0)-\zeta )^{-1}=\sum_{j=-\infty}^\infty \sum_{k=0}^\infty
\frac{\pi_{j,k}}{(j+1/2)^2 +\lambda_k -\zeta}, \quad\zeta\notin \text{Spec}(H_0(\theta_0)).
\]
For future reference let us observe that for $\zeta=\tau^2-i\rho\tau$ with $\rho\ge 1$, we have
\begin{equation}
\label{eq_im_zeta}
\text{Im}\sqrt{\zeta}=\sqrt{\frac{|\zeta|-\text{Re}\, \zeta}{2}}=\sqrt{\frac{|\tau|\sqrt{\tau^2+\rho^2}-\tau^2}{2}}=\frac{\rho}{2}+\mathcal{O}\bigg(\frac{1}{\tau^2}\bigg)\ge \frac{1}{4},
\end{equation}
provided that $|\tau|$ is large.

In the case when $\nu=0$, we have
\[
G_\tau \tilde f_0=R(\tau^2-i\tau) \tilde f_0,
\]
and thus, \eqref{eq_3_9} becomes the uniform resolvent estimate \eqref{eq_3_8} with $z^2=\tau^2-i\tau$. 

Let us show that in the case $\nu\ge 1$,   \eqref{eq_3_9} follows from the resolvent estimate
\begin{equation}
\label{eq_3_9_1}
\|R(\tau^2- i(2^\nu+1)\tau ) \tilde f_\nu \|_{L^{\frac{2n}{n-2}}(\mathbb{T}^n)}\lesssim \| \tilde f_\nu \|_{L^{\frac{2n}{n+2}}(\mathbb{T}^n)}
\end{equation}
where the implicit  constant is independent of $\tau$ and $\nu$. To that end, we write 
\[
(R(\tau^2- i(2^\nu+1)\tau ) -G_\tau  ) \tilde f_\nu=\sum_{j=-\infty}^\infty \sum_{k=0}^\infty a_{j,k,\nu}(\tau) \pi_{j,k} \tilde f_\nu,
\]
where 
\begin{align*}
a_{j,k,\nu}(\tau)=\frac{i\tau(2j-2^{\nu})1_{[2^{\nu-1},2^\nu)}(|j|)}{\big(\big(j+\frac{1}{2}\big)^2 + 2i\big(j+\frac{1}{2}\big)\tau -\tau^2+\lambda_k\big)\big(\big(j+\frac{1}{2}\big)^2 + i(2^\nu+1)\tau -\tau^2+\lambda_k\big)}.
\end{align*}
Using the fact that $\sum_{m=0}^\infty \chi_m^2=1$, and  spectral cluster estimates \eqref{eq_Sogges_estim_1} and \eqref{eq_Sogges_estim_2}, we get
\begin{equation}
\label{eq_3_10}
\begin{aligned}
\|(R(&\tau^2- i(2^\nu+1)\tau ) -G_\tau  ) \tilde f_\nu\|_{L^{\frac{2n}{n-2}}(\mathbb{T}^n)}\\
&\lesssim \sum_{m=0}^\infty (1+m)^{1/2} \|\chi_m(R(\tau^2- i(2^\nu+1)\tau ) -G_\tau  ) \tilde f_\nu\|_{L^2(\mathbb{T}^n)}\\
&\lesssim \bigg( \sum_{m=0}^\infty (1+m) \sup_{m\le \sqrt{\big(j+\frac{1}{2}\big)^2+\lambda_k}<m+1} |a_{j,k,\nu}(\tau)|\bigg) \|\tilde f_\nu\|_{L^{\frac{2n}{n+2}}(\mathbb{T}^n)}.
\end{aligned}
\end{equation}
To see that the above series converges and is bounded uniformly with respect to $\tau$ with $|\tau|\ge 1$ and $\nu$, we first observe using \eqref{eq_3_5_0} and \eqref{eq_3_6_0}  that 
\begin{align*}
& \sup_{m\le \sqrt{\big(j+\frac{1}{2}\big)^2+\lambda_k}<m+1} |a_{j,k,\nu}(\tau)|\\
& \lesssim \sup_{m\le \sqrt{\big(j+\frac{1}{2}\big)^2+\lambda_k}<m+1}
 \frac{2^\nu|\tau|}{\big( \big| \big(j+\frac{1}{2}\big)^2+\lambda_k -\tau^2\big| +2^{\nu-1}|\tau|\big)^2}
 \lesssim  \frac{2^\nu|\tau|}{ (m^2 -\tau^2)^2 +4^{\nu-1}\tau^2}.
\end{align*}
Hence, we have
\begin{equation}
\label{eq_3_11}
\begin{aligned}
 \sum_{m=0}^\infty (1+m) \sup_{m\le \sqrt{\big(j+\frac{1}{2}\big)^2+\lambda_k}<m+1} |a_{j,k,\nu}(\tau)|\lesssim 2^\nu|\tau|  \sum_{m=0}^\infty 
 \frac{1+m}{ (m^2 -\tau^2)^2 +4^{\nu-1}\tau^2}\\
 \lesssim \frac{1}{|\tau|} + 2^\nu|\tau|  \sum_{m=1}^\infty 
 \frac{m}{ (m^2 -\tau^2)^2 +4^{\nu-1}\tau^2}
 \lesssim 1+\int_{0}^\infty \frac{ 2^\nu|\tau|  t}{(t^2 -\tau^2)^2 +4^{\nu-1}\tau^2}dt\\
 \lesssim \int_{-\infty}^\infty \frac{ds}{s^2+1}<\infty.
\end{aligned}
\end{equation}
Here we have performed the change of variables $s=(t^2/|\tau|-|\tau|)/2^{\nu-1}$.
Thus, it follows from \eqref{eq_3_10} and \eqref{eq_3_11} that 
\[
\|(R(\tau^2- i(2^\nu+1)\tau ) -G_\tau  ) \tilde f_\nu\|_{L^{\frac{2n}{n-2}}(\mathbb{T}^n)}\lesssim \|\tilde f_\nu\|_{L^{\frac{2n}{n+2}}(\mathbb{T}^n)},
\]
with a constant uniform in $\tau$ and $\nu$. The estimate \eqref{eq_3_9} is therefore a consequence of the uniform resolvent estimate  \eqref{eq_3_9_1},  in view of  \eqref{eq_im_zeta}.
The proof of Proposition \ref{prop_carleman_L_n_2} is complete. 
\end{proof}

It is now easy to finish the proof of Theorem \ref{thm_main}.  Let  $q\in L^{\frac{n}{2}}(\mathbb{T}^n)$ and let us check that  there exists a constant $C>0$ such that for $\tau\in \R$ with $|\tau|$ sufficiently large, the following estimate holds, 
\begin{equation}
\label{eq_L_p_L_q_potential}
 \|u\|_{L^{\frac{2n}{n-2}}(\mathbb{T}^n)}\le  C \|( H_0(\theta_\tau)+q)u  \|_{L^{\frac{2n}{n+2}}(\mathbb{T}^n)},
\end{equation}
when $u\in \mathcal{D}(H(\theta_\tau))$. 
Let $u\in \mathcal{D}(H(\theta_\tau))$, then by  Lemma \ref{lem_domain}, we know that $u\in W^{2,\frac{2n}{n+2}}(\mathbb{T}^n)$. 
Denoting  by $C_0$ the uniform constant in the estimate \eqref{eq_L_p_L_p}, we write $q=q^\sharp+(q-q^\sharp)$ where $q^\sharp\in L^\infty(\mathbb{T}^n)$ and
\begin{equation}
\label{eq_3_1}
\|q-q^\sharp\|_{L^{\frac{n}{2}}(\mathbb{T}^n)}\le \frac{1}{4C_0}.
\end{equation}
By the embedding $L^2(\mathbb{T}^n)\hookrightarrow L^{\frac{2n}{n+2}}(\mathbb{T}^n)$ and the estimate \eqref{eq_L_p_L_2}, for  $\tau\in \R$ with $|\tau|\ge 1$ sufficiently large, we have
\begin{equation}
\label{eq_3_2}
\|q^\sharp u\|_{L^{\frac{2n}{n+2}}(\mathbb{T}^n)}\le C\|q^\sharp\|_{L^\infty(\mathbb{T}^n)} \|u\|_{L^2(\mathbb{T}^n)}
\le \frac{1}{2}
\| H_0(\theta_\tau) u  \|_{L^{\frac{2n}{n+2}}(\mathbb{T}^n)}.
\end{equation}
Using H\"older's inequality and estimates  \eqref{eq_L_p_L_p}, \eqref{eq_3_1}, and \eqref{eq_3_2}, for $\tau\in \R$ with $|\tau|\ge 1$ sufficiently large,  we have
\begin{align*}
\| (&H_0(\theta_\tau)+ q  ) u  \|_{L^{\frac{2n}{n+2}}(\mathbb{T}^n)}\ge 
\|  H_0(\theta_\tau) u  \|_{L^{\frac{2n}{n+2}}(\mathbb{T}^n)}
-\|q^\sharp u\|_{L^{\frac{2n}{n+2}}(\mathbb{T}^n)} \\
&- \|(q-q^\sharp) u\|_{L^{\frac{2n}{n+2}}(\mathbb{T}^n)} 
\ge \frac{1}{2}
\| H_0 (\theta_\tau)  u  \|_{L^{\frac{2n}{n+2}}(\mathbb{T}^n)} -\|q-q^\sharp\|_{L^{\frac{n}{2}}(\mathbb{T}^n)}\|u\|_{L^{\frac{2n}{n-2}}(\mathbb{T}^n)}\\
&\ge  \frac{1}{4C_0} \|u\|_{L^{\frac{2n}{n-2}}(\mathbb{T}^n)},
\end{align*}
which proves the estimate \eqref{eq_L_p_L_q_potential}.  This completes the proof of Theorem \ref{thm_main}.

\begin{appendix}

\section{Definition of operators using  sesquilinear forms, the Floquet theory and the Thomas approach}

\label{sec_1}

The material of this appendix is standard and is presented here for completeness  and convenience of the reader, see  \cite{Birman_Suslina_1999},  \cite{Reed_Simon_book_4},  \cite{Shen_2001}, \cite{Sjostrand_periodic},  \cite{Kuchment_book_1993}, \cite{Kuchment_Levendorskii_2002}, and \cite{Vesalainen_lic}.

\subsection{Definition of the Schr\"odinger operator, acting on $L^2(\R^n)$}

\label{sub_sec_1}
Let us start by reviewing the definition of the Schr\"odinger operator $H=-\Delta_g+q$ on $\R^n$, $n\ge 3$, with a potential  $q\in L^{n/2}_{\text{loc}}(\R^n; \C)$ and with a smooth Riemannian metric $g$, as a closed densely defined sectorial operator on $L^2(\R^n)$.  We assume that $g$ and $q$ are periodic with respect to the lattice $2\pi\Z^n$.

In what follows all norms are defined using the Riemannian volume element $|g|^{1/2}dx$. 

Consider  the sesquilinear form,    
\[
h[u,v]=\int_{\R^n} g^{jk}D_{x_k} u \overline{D_{x_j} v}|g|^{1/2} dx+\int_{\R^n} q u\overline{v}|g|^{1/2}dx,
\]
for $u,v\in C^\infty_0(\R^n)$. 

Let $(0,2\pi)^n$ be  the interior of the fundamental domain of the lattice $2\pi\Z^n$.  By H\"older's inequality and the Sobolev embedding $H^1((0,2\pi)^n)\subset L^{\frac{2n}{n-2}}((0,2\pi)^n)$, we get
\begin{align*}
\bigg|\int_{(0,2\pi)^n} q u\overline{v}|g|^{1/2}dx\bigg|&\le \| q\|_{L^{\frac{n}{2}}((0,2\pi)^n)}\|u\|_{L^{\frac{2n}{n-2}}((0,2\pi)^n)} \|v\|_{L^{\frac{2n}{n-2}}((0,2\pi)^n)}\\
&\le C \|u\|_{H^1((0,2\pi)^n)} \|v\|_{H^1((0,2\pi)^n)}.
\end{align*}
Replacing $(0,2\pi)^n$ by its translates $E_k=(0,2\pi)^n+k$, $k\in 2\pi\Z^n$, summing over all $k$, and using the Cauchy -- Schwarz inequality, we get 
\[
\bigg| \int_{\R^n} q u\overline{v}|g|^{1/2}dx \bigg| \le C \sum_{k\in 2\pi\Z^n}  \|u\|_{H^1(E_k)} \|v\|_{H^1(E_k)}\le C \|u\|_{H^1(\R^n)} \|v\|_{H^1(\R^n)}.
\]
 It follows that 
\begin{equation}
\label{eq_2_bounded_h}
|h[u,v]|\le C \|u\|_{H^1(\R^n)} \|v\|_{H^1(\R^n)},
\end{equation}
for $u,v\in C^\infty_0(\R^n)$. Hence, the form $h$ extends to a bounded sesquilinear form on $H^1(\R^n)$. 

We shall next check the coercivity of the form $h$ on $H^1(\R^n)$, i.e., 
\begin{equation}
\label{eq_2_coercivity_h}
\text{Re}\, h[u,u]\ge c_0\|u\|_{H^1(\R^n)}^2 -C_1 \|u\|_{L^2(\R^n)}^2,\quad  c_0>0,\quad C_1\in \R.
\end{equation}
Writing 
\begin{equation}
\label{eq_sec2_dec}
q= q^\sharp +(q-q^\sharp),
\end{equation} 
where $ q^\sharp\in L^\infty((0,2\pi)^n)$ and $\|q-q^\sharp\|_{ L^{\frac{n}{2}}((0,2\pi)^n)}\le \varepsilon$ for some $\varepsilon>0$ small, we have 
\begin{align*}
\int_{(0,2\pi)^n} |q||u|^2 |g|^{1/2}dx&\le \|q^\sharp \|_{L^\infty ((0,2\pi)^n)}\|u\|^2_{L^2((0,2\pi)^n)} \\
&+ \|q-q^\sharp\|_{L^{\frac{n}{2}}((0,2\pi)^n)} \|u\|^2_{L^{\frac{2n}{n-2}}((0,2\pi)^n)}\\
&\le \mathcal{O}_\varepsilon(1)  \|u\|^2_{L^2((0,2\pi)^n)}  + \mathcal{O}(\varepsilon)  \| u\|^2_{H^1((0,2\pi)^n)}.
\end{align*}
It follows  that
\begin{equation}
\label{eq_2_1}
\int_{\R^n} |q||u|^2 |g|^{1/2}dx \le \mathcal{O}_\varepsilon(1)  \|u\|^2_{L^2(\R^n)}  + \mathcal{O}(\varepsilon)  \| u\|^2_{H^1(\R^n)}.
\end{equation}
As $g$ positive definite, using \eqref{eq_2_1} and choosing $\varepsilon$ sufficiently small, we get \eqref{eq_2_coercivity_h}.

When equipped with the domain $\mathcal{D}(h)=H^1(\R^n)$, the form $h$  is densely defined, closed and  sectorial with the bound
\[
|\text{Im}\, h[u,u]| \le |h[u,u]| \le C c_0^{-1} (\text{Re}\, h[u,u]+C_1 \|u\|^2_{L^2(\R^n)}).
\]
Here we have used  \eqref{eq_2_bounded_h} and \eqref{eq_2_coercivity_h}.
 By \cite[Corollary 12.19]{Grubb_book}, there exists a closed densely defined sectorial operator $H$ on $L^2(\R^n)$, which we denote by $H=-\Delta_g+q$, with domain
\[
\mathcal{D}(H)=\{ u\in H^1(\R^n): (-\Delta_g+q) u\in L^2(\R^n) \},
\]
such that 
\[
h[u,v]=(Hu,v)_{L^2(\R^n)}, \quad u\in \mathcal{D}(H), \quad v\in \mathcal{D}(h).
\]

\subsection{Definition of the family of operators, acting on $L^2(\mathbb{T}^n)$}

\label{subsec_A2}
Let $\theta\in \C^n$. We shall next define a family of operators $H(\theta)$, acting  on $L^2(\mathbb{T}^n)$, where  $\mathbb{T}^n=\R^n/2\pi\Z^n$, formally given by
\begin{equation}
\label{eq_2_2_0}
H(\theta)=e^{-ix\cdot \theta} H e^{ix\cdot \theta}=|g|^{-1/2}(D_{x_j}+\theta_j)(|g|^{1/2} g^{jk}(D_{x_k}+\theta_k))+q.
\end{equation}
Let $u,v\in C^\infty(\mathbb{T}^n)$ and consider a family of sesquilinear forms,
\begin{equation}
\label{eq_2_2}
h(\theta)[u,v]=\int_{\mathbb{T}^n} g^{jk} D_{x_k}u\overline{D_{x_j}v} |g|^{1/2}dx+w[\theta][u,v],
\end{equation}
where 
\[
w[\theta][u,v]=\int_{\mathbb{T}^n} g^{jk} ( \theta_k u\overline{D_{x_j}v} +(D_{x_k}u)\theta_j \overline{v} +\theta_k\theta_j u\overline{v}) |g|^{1/2}dx
+\int_{\mathbb{T}^n}  qu\overline{v} |g|^{1/2}dx.
\]
By H\"older's inequality and the Sobolev embedding $H^1(\mathbb{T}^n)\subset L^{\frac{2n}{n-2}}(\mathbb{T}^n)$,
we see that  
\begin{equation}
\label{eq_2_bounded_h_theta}
|h(\theta)[u,v]|\le C \|u\|_{H^1(\mathbb{T}^n)}\|v\|_{H^1(\mathbb{T}^n)},
\end{equation}
with $C=C(\theta)>0$, 
and hence, for each $\theta\in \C^n$, the form $h(\theta)$ extends to a bounded sesquilinear form on $H^1(\mathbb{T}^n)$.

Using the Peter--Paul inequality,  the decomposition \eqref{eq_sec2_dec}, and the Sobolev embedding,  we get
\begin{align*}
|w(\theta)[u,u]|\le  \mathcal{O}(\varepsilon) \| u\|_{H^1(\mathbb{T}^n)}^2+ \mathcal{O}_{\theta,\varepsilon}(1) \|u\|^2_{L^2(\mathbb{T}^n)},
\end{align*}
for $\varepsilon>0$. Hence, using the fact that $g$ is positive definite, and choosing $\varepsilon>0$ sufficiently small in the previous estimate, we obtain that 
\begin{equation}
\label{eq_2_coercivity_h_theta}
\text{Re}\, h[\theta][u,u]\ge c_0\|u\|_{H^1(\mathbb{T}^n)}^2 -C_1 \|u\|_{L^2(\mathbb{T}^n)}^2,\quad  c_0>0,\quad C_1=C_1(\theta)\in \R.
\end{equation}

When equipped with the domain $\mathcal{D}(h(\theta))=H^1(\mathbb{T}^n)$, the form $h(\theta)$  is densely defined, closed and  sectorial with the bound
\begin{equation}
\label{eq_2_sectorial_h_theta}
|\text{Im}\, h(\theta)[u,u]| \le |h(\theta)[u,u]| \le C c_0^{-1} (\text{Re}\, h(\theta)[u,u]+C_1 \|u\|^2_{L^2(\mathbb{T}^n)}).
\end{equation}
Here we have used  \eqref{eq_2_bounded_h_theta} and \eqref{eq_2_coercivity_h_theta}, and we may also notice that the  constants $C$ and $C_1$ are uniform in $\theta$, on compact subsets of $\mathbb{C}^n$. 

By \cite[Corollary 12.19]{Grubb_book}, there exists a closed densely defined sectorial operator $H(\theta)$ on $L^2(\mathbb{T}^n)$, which we write as in \eqref{eq_2_2_0}, with domain
\begin{equation}
\label{eq_2_2_0_domain_new}
\mathcal{D}(H(\theta))=\{ u\in H^1(\mathbb{T}^n): H(\theta) u\in L^2(\mathbb{T}^n) \},
\end{equation}
such that 
\[
h(\theta)[u,v]=(H(\theta)u,v)_{L^2(\mathbb{T}^n)}, \quad u\in \mathcal{D}(H(\theta)), \quad v\in \mathcal{D}(h(\theta)).
\]
In view of \eqref{eq_2_2_0_domain_new} and \eqref{eq_2_2_0}, we have 
\begin{equation}
\label{eq_2_domain_H_theta_leading_term}
\mathcal{D}(H(\theta))=\{ u\in H^1(\mathbb{T}^n): (g^{jk}D_{x_j}D_{x_k}+q) u\in L^2(\mathbb{T}^n) \},
\end{equation}
and in particular we see that $\mathcal{D}(H(\theta))$ is independent of $\theta\in \C^n$.

Furthermore,  by \cite[Corollary 12.21]{Grubb_book} the spectrum of $H(\theta)$ is contained in the following angular set with opening $<\pi$:
\[
\{\lambda\in \C: \text{\Re}\,\lambda\ge -C_1(\theta)+c_0, |\text{Im}\, \lambda| \le C(\theta)c_0^{-1} (\text{Re}\,\lambda+ C_1(\theta))  \},
\] 
where the constants are taken from  \eqref{eq_2_coercivity_h_theta} and \eqref{eq_2_sectorial_h_theta}.

It follows that for each $\theta\in \C^n$,  the resolvent of $H(\theta)$ is compact on $L^2(\mathbb{T}^n)$, and hence, the spectrum of $H(\theta)$ is discrete, each eigenvalue having a finite algebraic multiplicity.    

We conclude that the family of operators $H(\theta)$ is an entire holomorphic family of type (A) with respect to each of the complex variables $\theta_1, \dots, \theta_n$, see \cite[Section VII.2]{Kato_book}.

\subsection{The Floquet decomposition} The idea here is to pass from the operator $H$, acting on functions on $\R^n$, to the family of operators $H(\theta)$, acting on functions on the torus $\mathbb{T}^n$.

To that end, for $u\in \mathcal{S}(\R^n)$,  we define the Floquet--Bloch--Gelfand transform by
\[
(Uu)(\theta,x)=e^{-ix\cdot\theta}\sum_{k\in \Z^n} e^{-2\pi k i \cdot\theta} u(x+2\pi k), \quad \theta\in \R^n,\quad x\in [0,2\pi]^n,
\]
where the series converges in the $C^\infty$--sense.  We have 
\[
U:\mathcal{S}(\R^n)\to C_{\text{Fl}}^\infty(\R^n_{\theta}\times \mathbb{T}^n_x),
\]
where 
\[
C_{\text{Fl}}^\infty(\R^n_{\theta}\times \mathbb{T}^n_x)=\{ f\in C^\infty(\R^n_{\theta}\times \mathbb{T}^n_x): 
f(\theta+l,x)=e^{-ix\cdot l} f(\theta,x),\  l\in \Z^n
\}.
\]

It follows that 
\[
\int_{(0,2\pi)^n}\int_{(0,1)^n} |(Uu)(\theta,x)|^2d\theta |g|^{1/2}dx= \sum_{k\in\Z^n} \int_{(0,2\pi)^n} |u(x+2\pi k)|^2|g|^{1/2}dx,
\]
and therefore,
\[
\|Uu\|_{L^2((0,1)^n_{\theta}\times (0,2\pi)^n_x)}=\|u\|_{L^2(\R^n)}.
\]
Hence, $U$ can be extended to an isometry,
\[
U: L^2(\R^n)\to L^2((0,1)^n_{\theta}\times \mathbb{T}^n_x).
\]
Thus, $U^*U=I$, where $U^*$ is the $L^2$--adjoint of $U$.  A direct computation shows that for $v\in C_{\text{Fl}}^\infty(\R^n_{\theta}\times \mathbb{T}^n_x)$, we have 
\begin{equation}
\label{eq_2_adjoint_U}
(U^*v)(x)=\int_{(0,1)^n} e^{ix\cdot\theta} v(\theta,x)d\theta\in  \mathcal{S}(\R^n).
\end{equation}

To see that $UU^*=I$, we first let $v\in C_{\text{Fl}}^\infty(\R^n_{\theta}\times \mathbb{T}^n_x)$. 
Then we have the Fourier series expansion with respect to $\theta$,
\[
v(\theta,x)=\sum_{k\in \Z^n} \bigg( \int_{(0,1)^n} v(\theta',x) e^{-i(2\pi k-x)\cdot\theta'} d\theta'\bigg)e^{i(2\pi k-x)\cdot\theta}.
\]
We get 
\[
U(U^*v)(\theta, x)=\sum_{k\in\Z^n} e^{-i(x+2\pi k)\cdot\theta} \int_{(0,1)^n} e^{i(x+2\pi k)\cdot\theta'} v(\theta', x)d\theta'=v(\theta, x),
\]
and therefore, by density, we have  $UU^*=I$ on $L^2((0,1)^n_{\theta}\times \mathbb{T}^n_x)$. 

Hence, the map 
\begin{equation}
\label{eq_2_unitary}
U: L^2(\R^n)\to L^2((0,1)^n_{\theta}\times \mathbb{T}^n_x)=L^2((0,1)^n_{\theta}; L^2(\mathbb{T}^n))=:\int^{\oplus}_{(0,1)^n} L^2(\mathbb{T}^n)d\theta,
\end{equation}
 is unitary. 
 
We shall next show that  
\begin{equation}
\label{eq_2_map_U_form_domain}
U: H^1(\R^n)\to L^2((0,1)^n; H^1(\mathbb{T}^n))
\end{equation}
is a linear homeomorphism. To that end, let  $u\in H^1(\R^n)$, and let us check that $Uu\in L^2((0,1)^n; H^1(\mathbb{T}^n))$.  Using the fact that for $u\in \mathcal{S}(\R^n)$, 
\[
D_{x_k}(Uu)(\theta,x)=(U(D_{x_k} u))(\theta,x)-\theta_k (Uu)(\theta,x),\quad k=1,\dots, n,
\]
we get 
\begin{equation}
\label{eq_2_der_Uu}
\|D_{x_k} (Uu)\|_{L^2((0,1)^n; L^2(\mathbb{T}^n))}\le \| D_{x_k} u\|_{L^2(\R^n)}+\|u\|_{L^2(\R^n)},\quad u\in \mathcal{S}(\R^n).
\end{equation}
By a standard approximation argument,  we get $Uu\in L^2((0,1)^n; H^1(\mathbb{T}^n))$, and the estimate \eqref{eq_2_der_Uu} extends to $u\in H^1(\R^n)$. Hence, the map \eqref{eq_2_map_U_form_domain} is continuous. 

It remains to show that the map \eqref{eq_2_map_U_form_domain} is surjective. Let $v\in L^2((0,1)^n; H^1(\mathbb{T}^n))$, and thus, $U^{-1}v\in L^2(\R^n)$. As a consequence of 
\eqref{eq_2_adjoint_U}, we get 
\[
\| D_{x_k}(U^{-1}v)\|_{L^2(\R^n)}\le \|D_{x_k} v\|_{L^2((0,1)^n; L^2(\mathbb{T}^n))} + \|v\|_{L^2((0,1)^n; L^2(\mathbb{T}^n))},
\]
for $v\in  C_{\text{Fl}}^\infty(\R^n_{\theta}\times \mathbb{T}^n_x)$. An approximation argument yields that $U^{-1}v\in H^1(\R^n)$, which shows the claim.

Now a direct computation shows that  
\begin{equation}
\label{eq_2_5}
h[u,v]=\int_{(0,1)^n} h(\theta)[(Uu)(\theta, \cdot), (Uv)(\theta, \cdot)]d\theta,
\end{equation}
for $u,v\in \mathcal{S}(\R^n)$. 
In view of the fact that the map \eqref{eq_2_map_U_form_domain} is a linear homeomorphism, and the forms $h$ and $h(\theta)$ are continuous on $H^1(\R^n)$ and $H^1(\mathbb{T}^n)$, respectively,   the decomposition \eqref{eq_2_5} extends to $u,v\in H^1(\R^n)$.

Let us show that we have  the following decomposition of the operator $UHU^{-1}$ in a direct integral, 
\begin{equation}
\label{eq_sec_2_direct_integral}
U H U^{-1}=\int_{(0,1)^n} ^{\oplus} H(\theta)d\theta.
\end{equation}
To that end let us recall the definition of the direct integral, 
\begin{align*}
\mathcal{D}\bigg( \int_{(0,1)^n} ^{\oplus} H(\theta)d\theta\bigg)=\bigg\{ \phi \in  \int^{\oplus}_{(0,1)^n} L^2(\mathbb{T}^n)d\theta: \phi(\theta,\cdot) \in \mathcal{D} (H(\theta))\text{ for a.a.}\\
\theta\in (0,1)^n,
\int_{(0,1)^n} \| H(\theta)\phi(\theta, \cdot)\|^2_{L^2(\mathbb{T}^n)}d\theta<\infty \bigg\},
\end{align*}
and 
\[
\bigg(\bigg(\int_{(0,1)^n} ^{\oplus} H(\theta)d\theta\bigg)\phi\bigg)(\theta,x)=(H(\theta) \phi(\theta,\cdot))(x),
\]
 for a.a. $\theta\in (0,1)^n$. Let us first prove that 
 \[
 \mathcal{D}(UH U^{-1})=\mathcal{D}\bigg( \int_{(0,1)^n} ^{\oplus} H(\theta)d\theta\bigg).
 \]
Using  \cite[Theorem 12.18]{Grubb_book}, \eqref{eq_2_unitary}, \eqref{eq_2_map_U_form_domain} and \eqref{eq_2_5}, we get
\begin{align*} 
\mathcal{D}(UHU^{-1})&=\{\phi\in L^2((0,1)^n; L^2(\mathbb{T}^n)): U^{-1}\phi\in \mathcal{D}(H)\}\\
&=\{\phi\in  L^2((0,1)^n; H^1(\mathbb{T}^n)):
 \exists f\in L^2(\R^n)\text{ s.t.}\\
&  h[U^{-1}\phi, \varphi]=(f,\varphi)_{L^2(\R^n)}, \forall \varphi\in H^1(\R^n)\}\\
 & =\{\phi\in  L^2((0,1)^n; H^1(\mathbb{T}^n)):
 \exists f\in L^2((0,1)^n; L^2(\mathbb{T}^n))\text{ s.t.}\\
&  \int_{(0,1)^n} h(\theta)[\phi, \psi] d\theta=(f,\psi)_{L^2((0,1)^n; L^2(\mathbb{T}^n))}, \forall \psi\in L^2((0,1)^n; H^1(\mathbb{T}^n))\}.
\end{align*}
Using \eqref{eq_2_2}, integrating by parts and modifying the function $f$ by a suitable expression depending only on the first order partial derivatives of the metric $g$ and the function $\phi$ in the variables $x\in \mathbb{T}^n$, we obtain that 
\begin{align*}
&\mathcal{D}(UHU^{-1})=\{\phi\in  L^2((0,1)^n; H^1(\mathbb{T}^n)):
 \exists f\in L^2((0,1)^n; L^2(\mathbb{T}^n))\text{ s.t.} \\
 &\int_{(0,1)^n}\int_{\mathbb{T}^n} ( D_{x_k}\phi \overline{D_{x_j}(g^{jk}|g|^{1/2}\psi)} +q \phi\overline{\psi}|g|^{1/2} )dx d\theta=\int_{(0,1)^n}\int_{\mathbb{T}^n} f\overline{\psi}|g|^{1/2}dxd\theta,\\
& \forall \psi\in L^2((0,1)^n; H^1(\mathbb{T}^n))
 \}\\
 &=\{ \phi\in  L^2((0,1)^n; H^1(\mathbb{T}^n)):( g^{jk}D_{x_j}D_{x_k}+q ) \phi \in L^2((0,1)^n; L^2(\mathbb{T}^n))\}\\
 &=\mathcal{D}\bigg( \int_{(0,1)^n} ^{\oplus} H(\theta)d\theta\bigg).
\end{align*}
Here we have also used \eqref{eq_2_domain_H_theta_leading_term}. 

Let $\phi\in \mathcal{D}(UHU^{-1})$. For any $\psi\in C^\infty_0((0,1)^n; C^\infty(\mathbb{T}^n))$, we get
\begin{align*}
(UHU^{-1} \phi, \psi)_{L^2((0,1)^n; L^2(\mathbb{T}^n))}&=h[U^{-1} \phi, U^{-1}\psi]=\int_{(0,1)^n} (H(\theta)\phi, \psi)_{L^2(\mathbb{T}^n)}d\theta\\
&=\bigg (\bigg(\int_{(0,1)^n} ^{\oplus} H(\theta)d\theta\bigg) \phi, \psi\bigg)_{L^2((0,1)^n; L^2(\mathbb{T}^n))},
\end{align*}
which shows  \eqref{eq_sec_2_direct_integral}. 

\subsection{Thomas's  approach}

To show that the operator $H$ has no eigenvalues, a fundamental idea of Thomas \cite{Thomas_1973} is to complexify the quasimomentum $\theta$ and use analytic perturbation theory.  As a consequence of this idea the following  result takes place. We have learned this result from \cite{Kuchment_book_1993}, \cite{Kuchment_Levendorskii_2002}, and for the proof we follow \cite{Vesalainen_lic}.

\begin{prop}
\label{prop_Floquet_Thomas}
Let $\lambda\in \C$.  Then if there exists $\theta=\theta(\lambda)\in \C^n$ such that the operator $H(\theta)-\lambda$, acting on $L^2(\mathbb{T}^n)$, has zero kernel,  then $\lambda$ is not an eigenvalue of the operator $H$, acting on $L^2(\R^n)$. 
\end{prop}

\begin{proof}
Seeking a contradiction, assume that $\lambda$ is  an eigenvalue of the operator $H$, acting on $L^2(\R^n)$, i.e. there exists $u\in \mathcal{D}(H)$, $\|u\|_{L^2(\R^n)}=1$, such that 
\begin{equation}
\label{eq_app_thomas_eigen}
(H-\lambda)u=0. 
\end{equation}

We shall show that $\lambda$ is  an eigenvalue of the operator $H(\theta)$, acting on $L^2(\mathbb{T}^n)$, for all $\theta\in \C^n$. Since the Floquet--Bloch--Gelfand transform $U$ is an isometry, we have 
\[
\|u\|^2_{L^2(\R^n)}=\int_{\theta\in (0,1)^n} \|(Uu)(\theta, \cdot)\|^2_{L^2(\mathbb{T}^n)} d\theta=1, 
\]
and therefore,
\begin{equation}
\label{eq_sec_2_4_1}
\mu_n(\{\theta\in (0,1)^n: \|(Uu)(\theta, \cdot)\|_{L^2(\mathbb{T}^n)}\ne 0\})>0.
\end{equation}
Here and in what follows $\mu_n$ is the Lebesgue measure on $\R^n$. 

It follows from \eqref{eq_app_thomas_eigen} that 
\[
(UHU^{-1} -\lambda) Uu=0.
\]
This together with the decomposition \eqref{eq_sec_2_direct_integral} implies that 
\[
\bigg( \int_{(0,1)^n}^\oplus H(\theta)d\theta -\lambda \bigg) Uu=0,
\]
and hence, 
\begin{equation}
\label{eq_sec_2_4_2}
(H(\theta) -\lambda) (Uu)(\theta,\cdot)=0
\end{equation}
for almost every $\theta\in (0,1)^n$. 

It follows from \eqref{eq_sec_2_4_1} and \eqref{eq_sec_2_4_2} that 
\[
\mu_n(\Theta)>0, \quad \Theta=\{ \theta\in (0,1)^n: \lambda \text{ is an eigenvalue of } H(\theta)\}.
\]

When $1\le j\le n$, let us consider the holomorphic family of operators 
\[
\mathcal{H}(\theta_j)= H(\theta_1^0, \dots, \theta_{j-1}^0,\theta_j,\theta_{j+1}^0,\dots,\theta_n^0), \quad \theta_j\in \C,
\]
while the complex values $\theta_k^0$, $k\ne j$, are kept fixed.  The resolvent of the operator $\mathcal{H}(\theta_j)$ is compact for each $\theta_j$, and an application of the analytic Fredholm theory, see \cite[Theorem VII. 1.10]{Kato_book}, allows us to conclude  that either $\lambda$ is an eigenvalue of the operator $\mathcal{H}(\theta_j)$ for each $\theta_j\in \C$ or the set of points $\theta_j\in \C$, for which $\lambda$ is an eigenvalue of $\mathcal{H}(\theta_j)$ is discrete.  

 Let $\tilde\theta\in \C^n$ be an arbitrary fixed vector and let us show that $\lambda$ is an eigenvalue of $H(\tilde\theta)$. 
We shall show this by induction. First, consider the set
\[
\Theta_2=\{ (\theta_2,\dots, \theta_n)\in (0,1)^{n-1}: \mu_1 (\{ \theta_1\in (0,1): (\theta_1,\theta_2,\dots,\theta_n)\in \Theta\})>0 \}.
\]
Thus, for any $(\theta_2,\dots, \theta_n)\in \Theta_2$, we have
\[
\mu_1 (\{ \theta_1\in (0,1): \lambda \text{ is an eigenvalue of } \mathcal{H}(\theta_1)\})>0.
\]
Hence, by the analytic Fredholm theory, we conclude that $\lambda$ is an eigenvalue of $H(\theta)$ for all $\theta_1\in \C$ and all $(\theta_2,\dots, \theta_n)\in \Theta_2$, and therefore, $\lambda$ is an eigenvalue of $H(\tilde \theta_1,\theta_2,\dots,\theta_n)$ for all  $(\theta_2,\dots, \theta_n)\in \Theta_2$.

As $\mu_n(\Theta)>0$, by Fubini's theorem, we have $\mu_{n-1}(\Theta_2)>0$.   Consider the set 
\[
\Theta_3=\{ (\theta_3,\dots, \theta_n)\in (0,1)^{n-2}: \mu_1 (\{ \theta_2\in (0,1): (\theta_2,\dots,\theta_n)\in \Theta_2\})>0 \}.
\]
Then for any $(\theta_3,\dots, \theta_n)\in \Theta_3$, since $\mu_1 (\{ \theta_2\in (0,1): (\theta_2,\dots,\theta_n)\in \Theta_2\})>0$, by the analytic Fredholm theorem, we get that $\lambda$ is an eigenvalue of the operator  $H(\tilde \theta_1,\theta_2,\theta_3,\dots,\theta_n)$ for all $\theta_2\in \C$ and all  $(\theta_3,\dots, \theta_n)\in \Theta_3$. In particular, $\lambda$ is an eigenvalue of   $H(\tilde \theta_1,\tilde \theta_2,\theta_3,\dots,\theta_n)$ for  all  $(\theta_3,\dots, \theta_n)\in \Theta_3$. Continuing in the same fashion after $n-2$ steps, we show that $\lambda$ is  an eigenvalue of   $H(\tilde \theta)$. This contradicts the assumption of the proposition. The proof  is complete. 
\end{proof}

\begin{rem}
\label{rem_app_selfadj}
In the case of a real valued periodic potential $q\in L^{\frac{n}{2}}_{\emph{\text{loc}}}(\R^n)$, the sesquilinear form $h$ is symmetric and bounded from below, and therefore, the Schr\"odinger operator $H=-\Delta_g+q$, acting on $L^2(\R^n)$, is self-adjoint and bounded from below.  For any $\theta\in (0,1)^n$, the  sesquilinear form $h(\theta)$ is symmetric and bounded from below, and thus, the  
operator $H(\theta)$, acting on $L^2(\mathbb{T}^n)$,  is self-adjoint  and bounded from below.  Furthermore, the resolvent $(H(\theta)+i)^{-1}$ is a real-analytic function of $\theta\in (0,1)^n$, and $(H(\theta)+i)^{-1}$ is compact for every $\theta\in (0,1)^n$. Then using a general result of  \cite{Filonov_Sobolev_2004} and \cite{Gerard_Nier_1998}, concerning the spectrum of the analytic direct integral \eqref{eq_sec_2_direct_integral}, we conclude that the singular continuous component of the spectrum of $H$ is empty, and the pure point spectrum is  at most discrete, consisting only of isolated points without finite accumulation points, and each eigenvalue $\lambda$ of $H$ is of infinite multiplicity. 
Hence, in the case of a real valued periodic potential $q$, the absence of eigenvalues implies that the spectrum of $H$ is purely absolutely continuous. 
\end{rem}
\end{appendix}

\section*{Acknowledgements}

The research of K.K. is partially supported by the
Academy of Finland (project 141075).  The research of
G.U. is partially supported by the National Science Foundation.

\end{document}